\documentclass[oneside,english]{amsart}

\usepackage[latin1]{inputenc}
\usepackage{amssymb}

\usepackage{amssymb}
\usepackage[]{epsf,epsfig,amsmath,amssymb,amsfonts,latexsym}

\newtheorem{thm}{Theorem}[section]
\newtheorem{lem}[thm]{Lemma}
\newtheorem{prop}[thm]{Proposition}
\newtheorem{cor}[thm]{Corollary}

\theoremstyle{definition}
\newtheorem{defn}[thm]{Definition}
\newtheorem{remark}[thm]{Remark}
\newtheorem{example}[thm]{Example}
\newtheorem{question}[thm]{Question}

\newcommand {\ZD}{\mathbb{Z}^d}
\newcommand {\ZZ}{\mathbb{Z}}
\newcommand {\NN}{\mathbb{N}}
\newcommand {\RR}{\mathbb{R}}

\newcommand {\GG}{\mathbb{G}}
\newcommand{\HH}{\mathbb{H}}

\newcommand {\Uu}{\mathcal{U}}

\newcommand {\Hom}{\mathit{Hom}}
\newcommand {\GHom}{\mathit{GHom}}
\newcommand {\dprime}{topologically direct prime}
\newcommand{\DPF}{\emph{DPF}}

\newcommand {\Ht}{\mathit{ht}}
\newcommand {\Aut}{\mathit{Aut}}
\title{Direct topological factorization for topological flows}

\author{Tom Meyerovitch}
\address{Tom Meyerovitch\\
Department of Mathematics\\
Ben-Gurion University of the Negev}
\email{mtom@math.bgu.ac.il}
\thanks{The research leading to these results has received funding from the People Programme (Marie Curie Actions) of the European Union's Seventh Framework Programme (FP7/2007-2013) under REA grant agreement no. 333598
 and from the Israel  Science Foundation (grant no. 626/14)}
\keywords{topological dynamics; symbolic dynamics; subshifts; topological entropy; Perron numbers; cocycle}

\subjclass[2000]{37B05,37B10, 37B50, 37B40}

\begin{document}
\maketitle

\begin{abstract}
This paper considers the general question of when a topological  action of a countable group can be factored
into a direct product of a nontrivial actions. In the early 1980's D. Lind considered such questions for $\ZZ$-shifts of finite type.
We study in particular  direct factorizations of subshifts of finite type over $\ZD$ and other groups, and $\ZZ$-subshifts which are not of finite type.
The main results concern direct factors of the multidimensional full $n$-shift, the multidimensional $3$-colored chessboard and the Dyck shift over a prime alphabet.

A direct factorization of an expansive $\GG$-action must be finite, but
a example is provided of a non-expansive $\ZZ$-action for which there is no finite direct prime factorization.
The question about existence of direct prime factorization of expansive actions remains open, even for $\GG=\ZZ$.
\end{abstract}

\section{Introduction}

In this paper we study the notion of ``direct-factorization'' for topological dynamical systems.  Other concepts of ``factorizations'' and ``disjointness''  in topological dynamics and ergodic theory have numerous, diverse and deep applications in mathematics,  going back at least  to  Furstenberg's influential paper  \cite{furstenberg_disjointness67}.

Throughout this paper, $\GG$ will denote a countable group.
By  a $\GG$-\emph{topological dynamical system} or $\GG$-\emph{flow} we mean a pair  $(X,T)$, where $X$ is a Hausdorff compact topological space, and $T$ is an action of the group $\GG$ on  $X$ by homeomorphisms. In other words, the map  $g \mapsto T^g$  is a  homomorphism from $\GG$ to the group of  self-homeomorphisms of $X$.
 A $\GG$-flow $(Y,S)$ is a \emph{factor} of another $\GG$-flow $(X,T)$ if there exists a continuous surjective map $\pi:X \to Y$ which is equivariant, meaning $ S^g \circ \pi = \pi \circ T^g$ for all $g \in \GG$. The map $\pi$ is called a \emph{factor map} or \emph{semi-conjugacy}.  A $\GG$-flow is called \emph{prime} if its only proper factor is the \emph{trivial $\GG$-flow}, which  is the trivial action on a singleton.

A \emph{direct topological factorization} of a $\GG$-flow $(X,T)$ is a \emph{topological conjugacy} or \emph{isomorphism} of the form  $$(X,T) \cong (Y_1 \times \ldots \times  Y_r, S_1 \times \ldots \times S_r).$$
In other words, a direct topological  factorization corresponds to a homeomorphism $\phi:X \to \prod_{i=1}^r Y_i$ so that $\phi (T^g) = (S_1\times \ldots \times S_r)^g \phi(x)$ for all $x \in X$, $g \in \mathbb{G}$.
 We say each $(Y_i,S_i)$ as above is a \emph{direct factor} of $(X,T)$. Call a $\GG$-flow $(X,T)$  \emph{topologically direct prime} if it does not admit a non-trivial direct topological factor.   By a \emph{direct prime factorization} (\DPF) we mean  a direct factorization of $(X,T)$ into direct-prime flows.
Any direct factor of a $\GG$-flow  is indeed a factor.  It follows that any  prime $\GG$-flow is direct prime, but the converse is false.

The existence of a non-trivial factorization for a flow $(X,T)$ immediately implies that the topological space $X$ is homeomorphic to a non-trivial product $Y \times Z$. Thus, for instance any flow on the unit interval $X=[0,1]$ or the one-dimensional sphere $X= \RR/\ZZ$ is direct-prime, for ``purely topological reasons'', having nothing to do with the group action.

Most of our study will concern \emph{symbolic flows} or \emph{subshifts}. These are \emph{expansive} flows on totally disconnected compact metrizable topological spaces. One of the motivations for our choice to focus on symbolic systems is the  attempt to avoid  ``purely topological obstructions''  as above.

 Here is an outline of the rest of the paper:

 In Section \ref{sec:expansive} we consider the following question: Given a $\GG$-flow, is it isomorphic to a finite product of direct-prime $\GG$-flows?
 In the expansive case, the question remains open. An example of a $\ZZ$-action for which there is no finite direct prime factorization is described.

 In Section \ref{sec:SFT_factorization} we consider direct factorizations for \emph{subshifts of finite type} (SFTs). We review results about $\ZZ$-SFTs and discuss factorization for $\ZD$-SFTs, where much less is known. We present a result regarding direct factorization of $\ZD$-full shifts. We also obtain a partial result which provides a weak form of a conjecture of J. Kari (Theorem \ref{thm:full_shift_factorizations}).

 In the remaining sections, we study direct-factorizations for specific systems. In Section \ref{sec:3CB_factorization} we consider specific $\ZD$-subshift of finite types: We prove that the $d$-dimensional $3$-colored chessboard is topologically direct-prime for any $d \ge 1$. In Section \ref{sec:dyck_factorization} we consider \emph{Dyck shifts}. These are  $\ZZ$-subshifts which are not of finite type. In both cases we  establish that the systems are topologically direct prime.
 Our methods involve specific combinatorial and algebraic structure of the systems. To obtain our result on the  $3$-colored chessboards we rely on the cohomology of the system. For the Dyck shifts we rely on lack of intrinsic ergodicity, and the structure of the measures maximal entropy. In both cases we exploit information about periodic points of the system.

\noindent\textbf{Acknowledgment:} I'd like to thank Mike Boyle, Brian Marcus and  Klaus Schmidt for valuable discussions, clarifying both historical and mathematical aspects, and the anonymous referee for valuable suggestions and corrections.

\section{On the existence of finite direct topological  factorizations}
\label{sec:expansive}

A $\GG$-flow $(X,T)$ is  called \emph{expansive} if there exist a finite open cover $\Uu = \{ U_1,\ldots,U_L\}$ so that for any function $F: \GG \to \Uu$  we have $\left|\bigcap_{g \in \GG}T^{-g}\left[ F(g)\right]\right| \le 1$.
Informally, this means that points can be separated by finite-precision measurements along the orbit.
To slightly simplify the proofs, we assume below that the topological  spaces involved are metrizable:
Let  $d$ be a compatible metric on $X$, then $(X,T)$ is expansive if and only if there exists an $\epsilon >0$ with the property
that for every pair $x,y$ of distinct points in $X$ there exists a $g \in \GG$ with $d(T^g(x),T^g(y)) \ge \epsilon$ \cite[Lemma $17.10$]{glasner_book}. Such $\epsilon$ is called an expansive constant for $(X,T)$, with respect to the metric $d$.
We note that the assumption that  $X$ is metrizable is not essential for any of the  results below.

The following question remains open, in particular when $\GG=\ZZ$:

\begin{question}
Is any expansive $\GG$-flow isomorphic to a finite product of direct-prime $\GG$-flows?
\end{question}

The following simple observation is useful for the study of direct factorizations of expansive systems:
\begin{prop}\label{prop:direct_factor_expansive}
Any direct factor of an expansive $\GG$-flow is  expansive.
\end{prop}
\begin{proof}
Suppose $(X,T) \cong (Y_1\times Y_2,S_1 \times S_2)$ is expansive.
Let $d_i$ be a compatible metric on $Y_i$.
Identifying $X$ with $Y_1 \times Y_2$,  it follows that $d=d_1+d_2$ is a compatible metric on $X$.
Let $\epsilon>0$ be an expansive constant for $(X,T)$ with respect to the metric $d$.
 fix  distinct points $y_1,\tilde y_1$   in $Y_1$ and $y_2 \in Y_2$, and consider $x = (y_1,y_2) $ and $\tilde x = (\tilde y_1,y_2)$ as points in $X$, whose projections to $Y_2$ coincide.
By expansiveness of $(X,T)$ there exists $g \in \GG$ so that $d(T^g(x),T^g(\tilde x)) > \epsilon$. Now:
$$d\left(T^g(x),T^g(\tilde x)\right) = d_1(S_1^g(y_1),S_1^g(\tilde y_1)) +  d_1(S_2^g(y_2),S_2^g(y_2)) =  d_1(S_1^g(y_1),S_1^g(\tilde y_1)) .$$
It follow that $\epsilon$ is also an expansive constant for $(Y_1,S_1)$ with respect to the metric $d_1$.

\end{proof}

The following remark was kindly brought to the author's attention by the anonymous referee.

\begin{remark}A factor of an expansive $\GG$-flow might not be expansive. For instance, an irrational rotation is a non-expansive factor of the corresponding Sturmian shift, which is expansive.  Even  when  $(Y,S)$ is a factor of $(X,T)$ and the factor map $\pi:X \to Y$ is open, expansivity of $(X,T)$ does not in general imply expansivity of $(Y,S)$. An example for an expansive algebraic action of the free group on two generators admitting a non-expansive open factor via an algebraic map is described \cite[Remark $3.4$]{chung_li_expansive_algebriac_action}. As remarked in \cite{chung_li_expansive_algebriac_action}, such algebraic examples are impossible in the case $\GG= \ZD$ by \cite[Corollary $3.11$]{MR1069512}.
\end{remark}

\begin{lem}\label{lem:direct_factor_expansive}
An infinite product of non-trivial systems is not expansive. Namely,
if $(Y_i,S_i)_{ i \in \mathbb{N}}$ are a sequence of non-trivial flows, then their product $\prod_{i =1}^\infty(Y_i,S_{i})$ is not expansive.
\end{lem}

\begin{proof}
Suppose
$(X,T) = \prod_{i =1}^\infty(Y_i,S_{i})$ is expansive. 
As in the proof of Proposition \ref{prop:direct_factor_expansive}, suppose $d_i$ is a compatible metric on $Y_i$. Since $Y_i$ is compact, the diameter of $Y_i$ with respect to $d_i$ is bounded, so by rescaling the distance $d_i$ we can assume $\sup\left\{ d_i(y,z)~:~ y,z \in Y_i\right\} \le 1$ for all $i \in \mathbb{N}$.
Then  $d= \sum_{i=1}^\infty 2^{-i} d_i$ is a compatible metric on $X$. If $Y_i$ is non-trivial for infinitely many $i$'s then for any $n \in \mathbb{N}$ we can find distinct $x,\tilde x \in X$, with $x= (y_i)_{i \in \mathbb{N}}$ and
$\tilde x= (\tilde y_i)_{i \in \mathbb{N}}$ so that $y_i = \tilde y _i$ for $i \le n$. It follows that $0 < d(T^g(x),T^g(\tilde x)) < 2^{-n}$ for all $g \in \GG$, so $(X,T)$ is not expansive.
\end{proof}

\begin{lem}\label{lem:no_factozation_infinite_product}
If $(X,T)$ does not admit a finite \DPF~ then $(X,T)$ admits a factor which is isomorphic to an infinite product of non-trivial flows.

\end{lem}
\begin{proof}
By induction on $n \in \mathbb{N}$, construct sequences $\{(Y_n,S_n)\}_{n=1}^\infty$ and $\{(X_n,T_n)\}_{n=1}^\infty$ of non-trivial $\GG$-flows
so that $(X_1,T_1)$ and $(X_{n},T_{n}) \cong (Y_n \times X_{n + 1}, S_n \times T_{n+1})$ and $(X_n,T_n)$ does not admit a finite \DPF~.
Indeed, assuming $(X_n,T_n)$ does not admit a \DPF, in particular it is not direct-prime, so it is a product of two non-trivial $\GG$-flows,  at least one of which does not admit a finite \DPF.

Let $\pi_n:X \to Y_n$ be a factor map from $(X,T)$ to $(Y_n,S_n)$. The map $\pi:X \to \prod_{n=1}^\infty Y_n$ is continuous and equivariant, and the image is dense in $\prod_{n=1}^\infty Y_n$. Because $X$ is compact it follows that $\pi(X)=\prod_{n=1}^\infty Y_n$ . So $(\prod_{n=1}^\infty Y_n,\prod_{n=1}^\infty S_n)$ is a factor of $(X,T)$.
\end{proof}

From Lemma \ref{lem:no_factozation_infinite_product} we obtain the following corollary:
\begin{cor}
If $(X,T)$ is a $\GG$-flow such that $X$  is countable set, then $(X,T)$ admits a finite \DPF.
\end{cor}

\begin{question}What $\GG$-flows admit a finite \DPF?
\end{question}
Example \ref{exmp:odometer} below is a system which does not admit a finite \DPF~, but is  isomorphic to an infinite product of non-trivial direct-prime systems.
We do not know if being  isomorphic to an infinite product of non-trivial systems precludes the possibility of a  finite \DPF.

We conclude this section by discussing  direct factorizations of group rotations, providing an example for a (non-expansive) system which does not admit a finite \DPF.

\subsection{Direct factorization of group rotations}

A $\ZZ$-dynamical system  $(X,T)$ is called a \emph{group rotation} if $X$ admits a commutative group structure compatible with the given compact topology and $T(x)=x+x_0$ for some $x_0 \in X$. For a group rotations $(X,T)$, if the orbit of some (hence any) $x \in X$ is dense then $(X,T)$  is uniquely ergodic, where the unique invariant measure is Haar measure.

Recall that a $\GG$-action is \emph{equicontinuous} for every $\epsilon >0$ there exists $\delta >0$ such that $d(x_1,x_2) < \delta$ implies $d(T^g(x),T^g(y))< \epsilon$ for every $g \in \GG$.

For a group $\Gamma$ a  direct group factorization is a group isomorphism $\Gamma \cong \Gamma_1 \times \Gamma_2$.
\begin{prop}\label{cor:compact_group_direct_factorization}
The direct topological factorizations of a minimal group rotation $(X,T)$ are in bijection with the direct factorizations of $X$ as a topological group.
\end{prop}
\begin{proof}
A well-known  characterization of minimal group rotations states that a minimal $\ZZ$-dynamical $(X,T)$ system is a group rotation iff it is equicontinuous. Furthermore, any automorphism of $(X,T)$ as a $\ZZ$-flow respects the group structure on $X$  \cite[Theorem $1.8$]{glasner_book}.
Because a factor of an equicontinuous system is also equicontinuous, it follows that any factor of a minimal group rotation is a minimal group rotation.
\end{proof}

\begin{example}\label{exmp:odometer}
Let $X_{\mathbb{P}}= \prod_{p \in \mathbb{P}} \ZZ / p\ZZ$ where $p$ ranges over all primes $\mathbb{P}$
, and Let $T:X \to X$ be the map
given by $(Tx)_p = x_p +1 \mod p$. This is a minimal compact group rotation,  uniquely ergodic with rational pure-point spectrum. Such systems are often called ``Odometers''.
\end{example}

\begin{prop}\label{prop:primes_odomoter}
The Odometer $(X_{\mathbb{P}},T)$ does not admit a finite direct factorization into direct-primes.
\end{prop}

\begin{proof}
By Corollary \ref{cor:compact_group_direct_factorization} the direct-topological factorizations of $(X_{\mathbb{P}},T)$ correspond to the direct-group factorizations of the group
$X_{\mathbb{P}}=\prod_{p \in \mathbb{P}} \ZZ / p\ZZ$, but this group does not admit a finite direct group factorization into direct-prime groups.
To see this note that in any direct factorization $X_{\mathbb{P}}= \prod_{k=1}^n Y_k$ of $X_{\mathbb{P}}$, for each prime $p$ there is a unique $k \in \{1,\ldots,n\}$ such that $Y_k$ has non-trivial $p$-torsion, and the $p$-torsion of $Y_k$  is isomorphic to $\ZZ / p\ZZ$. It follows that each $Y_k$ is isomorphic as a group to $\prod_{p \in A_k} \ZZ / p\ZZ$, where $\bigcup_{k=1}^n A_k = \mathbb{P}$, thus at least one direct factor $Y_k$ is not prime.
\end{proof}

\section{Direct factorizations for shifts of finite type}
\label{sec:SFT_factorization}

A \emph{subshift} over a countable group $\GG$ is characterized as an expansive $\GG$-flow $(X,T)$ of a totally-disconnected compact metrizable space $X$. A more concrete description is the following: A $\GG$-system $(X,T)$ is a subshift iff it is isomorphic to a subsystem of $(A^\GG,\sigma)$ where $A$ is a finite set, called ``the alphabet'' or ``spins'', and $\sigma$ is the \emph{shift-action}, given by
$(\sigma^g(x))_h = x_{g^{-1} h}$. 
It is well known and easy to check that any subsystem of  $(A^\GG,\sigma)$ can be specified by a countable set
$\displaystyle \mathcal{F} \subset \bigcup_{ F \Subset \GG} A^F$ of ``forbidden configurations'' as follows:
\begin{equation}\label{eq:subshift_forbidden_words}
X_\mathcal{F} = \left\{ x \in A^\GG ~:~ (\sigma^g x)_F \not \in \mathcal{F} ~,~ \forall g \in \GG ,~ F \Subset \GG\right\},
\end{equation}
where $x_F \in A^F$ denotes the restriction of $x\in A^\GG$ to $F$, and $F \Subset \GG$.

The system $(A^\GG,\sigma)$ is called  the $\GG$-\emph{full-shift} over the alphabet $A$.

\begin{prop}(see  \cite[Section $6$]{lind_entropies_markov_shifts_1984})
Any direct factor of a $\GG$-subshift is a $\GG$-subshift.
\end{prop}
\begin{proof}
By Proposition \ref{prop:direct_factor_expansive} any direct factor of an expansive $\GG$-action is expansive.

Both compactness and metrizability pass to continuous images of topological spaces.

A subspace of a totally disconnected space is totally disconnected.
Since $Y$ is homeomorphic to $Y \times \{z_0\}$ for $z_0 \in Z$, it follows that in the case $Y \times Z$ is totally disconnected $Y$ and $Z$ must also be.
Thus, whenever  a totally-disconnected compact metrizable space $X$ is homeomorphic to $Y \times Z$,
$Y$ and $Z$ are also  totally-disconnected, metrizable and compact.
\end{proof}

A $\GG$-subshift is of \emph{finite type} (abbreviated $\GG$-\emph{SFT}) if it is isomorphic to a subsystem $(X_\mathcal{F},\sigma)$ of the form \eqref{eq:subshift_forbidden_words} with $|\mathcal{F}|< \infty$. Equivalently, a $\GG$-SFT is a $\GG$-subshift which is not isomorphic a strictly decreasing countable intersection of subshifts.

We record the following observation:
\begin{prop}\label{prop:SFT_direct_factor}
Any direct factor of a $\GG$-SFT is a $\GG$-SFT.
\end{prop}
\begin{proof}
Suppose $X \cong Y \times Z$  and $Z$ is not an SFT. Then there exist subshifts $$\ldots \subset Z_n \subset Z_{n-1} \subset \ldots \subset Z_1$$
such that $Z = \bigcap_{n=1}^\infty Z_n$ and
each inclusion is strict.
It follows that
$$Y \times Z = \bigcap_{n=1}^\infty \left( Y \times Z_n\right),$$ so $X$ is not an SFT.
\end{proof}

\begin{remark} A proof of Proposition  \ref{prop:SFT_direct_factor}  above for the particular case $\GG= \ZZ$ appears in  \cite[Section $6$]{lind_entropies_markov_shifts_1984}, using an argument involving ''canonical coordinates'' in the sense of Bowen \cite{bowen_top_entropy_axim_A}. It is not clear if there is a  meaningful extension of this notion for subshifts over general groups.
\end{remark}

\begin{remark}
 A subshift factor of an SFT is called a \emph{sofic shift}, and is not in general an SFT. When $\GG=\ZZ$ there are several useful equivalent conditions for a subshift to be sofic. For a general group  $\GG$ (or even $\GG=\ZD$ with $d>1$), sofic shifts are
generally not so well understood.
\end{remark}

Direct factorizations for  $\ZZ$-SFTs were considered and studied in \cite{lind_SFT_factorization_1983} and \cite{lind_entropies_markov_shifts_1984}. For $\ZZ$-SFTs,
it turns out that direct factorizations are intimately related a numerical invariant called the \emph{topological entropy}.

Recall that  a countable group $\GG$ is called \emph{amenable} if there exists a sequence $F_1,F_2 ,\ldots,F_n,\ldots \subset \GG$ of finite sets satisfying $\lim_{n \to \infty}\frac{|gF_n \bigtriangleup F_n|}{|F_n|}=0$ for all $g \in \GG$. A sequence as above is called a
\emph{F\"{o}lner sequence} in $\GG$.
This is one of many equivalent definitions for amenability.

The topological entropy of a $\GG$-subshift $(X,T)$ over an amenable group $\GG$ is given by
\begin{equation}\label{eq:amenable_subshift_entropy}
h(X,T) = \lim_{n \to \infty} \frac{\log \left| \left\{ x_{F_n}~:~ x \in X\right\}\right| }{|F_n|},
\end{equation}
where $(F_n)_{n=1}^\infty$  is a F\"{o}lner sequence.
When $\GG=\ZD$, balls with radius increasing to infinity form a  F\"{o}lner sequence. for instance:
$$F_n= \left\{ v \in \ZD ~:~ \|v\|_{\infty} \le  n \right\}.$$
The limit in \eqref{eq:amenable_subshift_entropy} is equal to the infimum of the sequence inside the limit and does not depend on the particular choice of F\"{o}lner sequence \cite[Theorem $4.9$]{MR2052281}. 
 More importantly, the topological entropy of $(X,T)$ is invariant under isomorphism, and does not depend on the representation of $(X,T)$.

In \cite{lind_SFT_factorization_1983} D. Lind  formulated  a characterization of the numbers which can be realized as entropies of $\ZZ$-SFT's and of topologically mixing (or aperiodic) SFTs.
Following \cite{lind_SFT_factorization_1983}, an algebraic integer $\lambda \in \RR_+$ is called a \emph{Perron number} if $\lambda$ is greater than the absolute value of any one of its algebraic conjugates. 

\begin{thm}\label{thm:lind_Z_SFT_entropy_perron} (Lind, \cite[Theorem $1$]{lind_SFT_factorization_1983})
For any Perron number $\lambda$ there exists a topologically mixing $\ZZ$-SFT $(X,T)$ such that $h(X,T)=\log(\lambda)$. Conversely, the topological entropy of any  mixing $\ZZ$-SFT $(X,T)$ is of the form $\log(\lambda)$ for a Perron number.
\end{thm}

For $\ZZ$-SFTs which are not necessarily  topologically mixing, the class of entropy numbers consists of logarithms of $n$-th roots of Perron numbers.

A Perron number $\lambda$ is called \emph{irreducible} if it impossible to write $\lambda=\alpha \beta$ with $\alpha,\beta >1$ Perron numbers.

\begin{thm}
(Lind  \cite[Theorem $4$ ]{lind_entropies_markov_shifts_1984})
Any perron number admits a finite number of factorizations into a finite number of irreducible Perron numbers. There are only a finite number of such, but factorization is not always unique.
\end{thm}

It follows that  a mixing $\ZZ$-SFT $(X,T)$ with $h(X,T)=\log(\lambda)$ with $\lambda$ an irreducible Perron number is direct-prime. 
There are additional obstructions to factorization of $\ZZ$-SFTs. For instance, a non trivial direct factorization can be detected by the dimension-module, which is a certain ordered abelian group along with an order preserving automorphism \cite{boyle_marcus_trow_resolving_maps}. 

For $\ZD$ with $d \ge 2$, there is no analogous condition on the entropy of $h(X,T)$ of an SFT which guarantees $(X,T)$ is direct-prime: The class of numbers which occur as the topological entropy for a $\ZD$-SFT is the class of non-negative right recursively enumerable numbers \cite{hochman_meyerovitch_entropies_2010}.

\begin{question} Does every $\ZD$-SFT admit only a finite number of \DPF s?
\end{question}

For $n \in \mathbb{N}$ we denote the first $n$ positive integers by
$$[n] = \{1,\ldots,n\}.$$

For $d=1$ the following result appears as Theorem $7$ of \cite{lind_entropies_markov_shifts_1984}:
\begin{thm}\label{thm:full_shift_direct_prime}
For any $d \ge 1$, the  full-shift $([n]^{\ZD},\sigma)$ is direct prime iff $n$ is a prime number.
\end{thm}

In fact, in \cite{lind_entropies_markov_shifts_1984}, three different proofs are provided for this fact for $d=1$. One proof is based on factorizations and does not extend to $d \ge 1$, since for every $d >1$ and a every prime number $p$ there exists a $\mathbb{Z}^d$ shift of finite type $X$ with topological entropy $h(X)=\log p$ which is not topologically prime. Furthermore, such examples can be constructed with ``good mixing properties'' \cite{MR2645044}.
 Another proof involves  ``$\otimes$-factorization'' of $\zeta$-functions and factorizations over $\mathbb{C}[t]$, and does not seem to extend to higher dimensions.
Yet another proof, attributed to G. Hansel uses only periodic point counts, and extends to any dimension with minor modification.

The essence of Hansel's proof for Theorem \ref{thm:full_shift_direct_prime} will serve us in  Section \ref{sec:dyck_factorization} as a component in the proof the Dyck shifts with a prime number of brackets are direct prime. We can deduce  Theorem \ref{thm:full_shift_direct_prime} using Theorem \ref{thm:full_shift_factorizations} below.

If $n=p_1 \cdot \ldots \cdot p_k$ is a factorization of $n$, it is clear that $$([n]^\GG,\sigma) \cong \prod_{i=1}^k ([p_i]^\GG,\sigma)$$ is a direct topological factorization, where $\GG$ can be any discrete group.

The following question seems to be open even for $\GG=\ZZ$.
\begin{question}
For which groups $\GG$ and natural numbers $n$ does $([n]^{\GG},\sigma)$ have a unique direct prime factorization upto reordering?
\end{question}

It is known that any direct factor of
$([n]^\ZZ,\sigma)$ is  \emph{shift-equivalent}
to $([m]^\ZZ,\sigma)$ with $ m | n$ (see Lemma $2.1$ of \cite{kari_CA_block_1996}).
Recall that $\ZZ$-shifts of finite type $(Y,\sigma)$ and $(Z,\sigma)$ are shift-equivalent if and only if they are \emph{eventually conjugate}, which means that $(Y,\sigma^k)$ is topologically conjugate to $(Z,\sigma^k)$
for all but finitely many $k$'s.

The following is a  natural generalization of ``eventual conjugacy'': Say that $\ZD$-flows $(X,T)$ and $(Y,S)$ are eventually conjugate if the subactions obtained from  $T$ and $S$
by passing to a finite-index subgroup $L < \ZD$ are topologically conjugate for all but a finite number of  subgroups $L$.
In \cite{kari_CA_block_1996} Kari conjectured that for any $d \ge 1$, and $n \in \NN$ and any direct factorization $([n]^{\ZD},\sigma) \cong (X \times Y, T \times S)$
there exists a finite index subgroup $L < \ZD$ so that
the restriction of the actions $T$ and $S$ to $L$ are both topologically conjugate to $\ZD$-full shifts.

Call $\GG$-actions $(X,T)$ and $(Y,S)$ \emph{periodically equivalent} if for any finite index normal subgroup $\HH \lhd \GG$, the number of $\HH$-fixed points in $X$ is equal to the number of $\HH$-fixed points in $Y$.

\begin{thm}\label{thm:full_shift_factorizations}
For any $d \ge 1$, any direct factor of a $\ZZ^d$ full-shift is periodically equivalent to a $\ZD$ full-shift.

Specifically:
Up to reordering of the terms, any direct-factorization of the $\ZZ^d$ full-shift   $([n]^{\ZD},\sigma)$ into direct-primes is of the form
$$([n]^{\ZD},\sigma) \cong \prod_{i=1}^k (Y_i,\sigma_i),$$
where $(Y_i,\sigma)$ is periodically equivalent to  the full-shift $([m_i]^{\ZD},\sigma)$  $n = \prod_{i=1}^k m_i$.
\end{thm}

\begin{proof}
We will prove the statement by induction on $d$.
The base cases $d=1$ follows from $\otimes$-factorizations of the  $\zeta$ function of the full shift as in \cite[Section $7$ ]{lind_entropies_markov_shifts_1984}, or using the fact that
shift equivalence determines periodic-equivalence as in  \cite[ Lemma $2.1$]{kari_CA_block_1996}.

We now assume $d>1$. 
Given a $\ZD$-dynamical system $(X,T)$ and an infinite subgroup $L < \ZD$, 
 we denote by $X^{(L)}$ the fixed points of the $L$-subaction of $T$:
$$X^{(L)} := \left\{ x \in X ~:~ T^n(x) =x ~\forall n \in L\right\}.$$
It follows that $(X^{(L)},T)$ is a sub-system of $(X,T)$. Furthermore, $\ZD/L$ acts on $X^{(L)}$ via $T$, so we interpret  $(X^{(L)},T)$ as a $(\ZD/L)$-flow.

Suppose $([n]^{\ZD},\sigma) \cong \prod_{i=1}^k (Y_i,T_i)$.
Let $m_i$ be the number of fixed points of $(Y_i,T_i)$. It follows that $n =\prod_{i=1}^k m_i$.
We need to show that for any finite-index subgroup $L < \ZD$, we have $|Y_i^{(L)}|=m_i^{[\ZD:L]}$.

Let  $L$ be a finite-index subgroup of $\ZD$. There exist $v_1,\ldots,v_d \in \ZD$ so that $L= \bigoplus_{i=1}^d \ZZ v_i$ and so that $\{v_1,\ldots,v_d \}$  is a basis for $\mathbb{Q}^d$ as a vector space over $\mathbb{Q}$.
Let $K :=\mathit{span}_{\mathbb{Z}} \{ v_1,\ldots, v_{d-1} \}$ ,
$H:=\mathit{span}_{\mathbb{Q}}\{v_1,\ldots,v_{d-1}\}$ and  $\tilde{K} := H \cap \ZD$.

Check that $\ZD/\tilde K$ is a torsion free quotient of $\ZD$, and $\dim_{\mathbb{Q}} (\mathbb{Q}^d / H) =1$. It follows that $\ZD/\tilde K \cong \ZZ$. Let   $w_d \in \ZD$ be such that $w_d + \tilde K$ generates the group $\ZD/\tilde K$. We see that $\ZD = \tilde K \oplus \ZZ w_d $.

By the discussion above, the group $\ZD/\tilde K$ acts on $(Y_i)^{(\tilde K)}$. This is a  direct factor of $X^{(\tilde K)}$, which can be viewed as a $\ZZ$-full shift because $\ZD/ \tilde K \cong \ZZ$. Observe that the points in $(Y_i)^{(\tilde K)}$  which are fixed  by the shift action of $\ZD/\tilde K$  are precisely the fixed point of $Y_i$ under the $\ZD$ action.
Viewing  $(Y_i)^{(\tilde K)}$ as a direct factor of  $X^{(\tilde K)}$, it follows by the induction hypothesis (case $d=1$) that $((Y_i)^{(\tilde K)})^{(w_d)}$ is periodically equivalent to  a $\ZZ$-full shift of the form $[m_i]^\ZZ$. In particular,
\begin{equation}\label{eq:ZD_full_shift_factor_1}
|((Y_i)^{(\tilde K)})^{(\langle  v_d \rangle )}| = m_i^{[ (\ZD/ \tilde{K}) : \langle v_d +  \tilde{K}\rangle]},
\end{equation}
where $ \langle v_d +  \tilde{K}\rangle$ is the subgroup of $\ZD/ \tilde{K}$ spanned by $ v_d + \tilde{K}$.
Viewing  $Y_i^{(\langle  v_d \rangle)}$ as a subshift with respect to the action of  $\tilde{K}\cong \ZZ^{d-1}$ we have:
\begin{equation}\label{eq:ZD_full_shift_factor_2}
|((Y_i)^{(\langle  v_d \rangle)})^{(\tilde{K})}|= |((Y_i)^{(\tilde K)})^{(\langle v_d +  \tilde{K}\rangle)}| = m_i^{[ (\ZD/ \tilde{K}) : \langle v_d +  \tilde{K}\rangle]},
\end{equation}
where the second inequality follows from \eqref{eq:ZD_full_shift_factor_1}.

Because  both $Y_i^{(\langle  v_d \rangle)}$ and $X^{(\langle  v_d \rangle)}$ are $\tilde{K}\cong \ZZ^{d-1}$-subshifts,  by induction hypothesis, $((Y_i)^{(\langle  v_d \rangle)})$ is periodically equivalent to a  $\tilde{K}$-full-shift.
It follows that
$$ |Y_i^{(L)}|= |((Y_i)^{(v_d)})^{(K)}|= |((Y_i)^{(v_d)})^{(\tilde{K})}|^{[\tilde K : K]} = m_i^{[ (\ZD/ \tilde{K}) : \langle v_d +  \tilde{K}\rangle]\cdot [\tilde K :K ] }.$$

Since $L \cong K \oplus \langle v_d \rangle$ we have
$[ (\ZD/ \tilde{K}) : \langle v_d +  \tilde{K}\rangle]\cdot [\tilde K :K ] = [\ZD : L],$
and so
$$ |Y_i^{(L)}|= m_i^{[\ZD : L]}.$$

\end{proof}

\noindent\textbf{Problem:} Do Theorems \ref{thm:full_shift_direct_prime} and \ref{thm:full_shift_factorizations} above extend to other countable groups?

The following example shows that for finite cyclic groups this is not the case. I thank the referee for highlighting  the relation to algebra.

\begin{example}\label{exmp:full_factor}
Consider the full-shift with $5$ symbols over the finite cyclic group $\GG= \ZZ/ 4\ZZ$. 
The orbit counts determine a finite $\GG$-flow up to isomorphism. The given system has $5$ fixed points, $10$ orbits of length $2$ and $150$ orbits of length $4$.
This $\GG$-action lifts to a $\ZZ$-action given by multiplication by $x$ on $L=\mathbb{F}_5[x]/\langle x^4 -1\rangle$. 
As an $\mathbb{F}_5[x]$-module this system is a direct sum of two submodules, corresponding to the factorization $x^2-1=(x^2-1)(x^2+1)$.  Explicitly, we have $L=M\oplus N$ where $M=(x^2-1)L$ and $N=(x^2+1)L$. 
Here $M$ has $1$ fixed point and $10$ orbits of length $4$, while $N$ has $5$ fixed points and $10$ orbits of length  $2$.
\end{example}

\section{Direct-primeness for the $3$-colored chessboard}
\label{sec:3CB_factorization}

In this section we prove that the  ``$d$-dimensional  3-colored chessboard'', denoted by $C_3^d$,  is direct-prime. 
The $\ZD$-subshift $C_3^d \subset \{0,1,2\}^{\ZD}$
is a subshift of finite type  which consists of proper 3-colorings of $\ZD$, where we consider $\mathbb{Z}^d$ as the vertices of the Cayley graph with respect to the usual generators. Namely:

$$C_3^d := \{ x \in \{0,1,2\}^{\mathbb{Z}^d}~:~ x_n \ne x_m \mbox{ whenver } \|n-m\|_1=1\},$$
where $\|m\|_1= \sum_{i=1}^d |m_i|$ is the $l^1$ norm of $m = (m_1,\ldots,m_d) \in \ZD$.

It is useful to interpret $C_3^d$ in the context of \emph{graph homomorphisms}. We introduce some notation:

Given  graphs $G=(V_G,E_G)$ and $H=(V_H,E_H)$ we let $\GHom(G,H)$ denote the space of graph homomorphisms from $G$ to $H$.
$$\GHom(G,H) ~:=~ \left\{ x \in (V_H)^{V_G}~:~ (g_1,g_2) \in E(G) \Longrightarrow (x_{g_1},x_{g_2}) \in E_H  \right\}.$$
When $H$ is finite, we consider  $\GHom(G,H)$ as a compact topological space, with the topology induced from the product topology on
$(V_H)^{V_G}$. 

We identify $\ZD$ with the vertices of the Cayley graph of $\ZD$ with respect to the natural set of generators, and
interpret  $A \subset \ZD$ as the vertex set of the induced graph from the Cayley graph of $\ZD$. Let $\GHom(A,H)$  denote the set of graph homomorphisms from $A$ to $H$. The restriction
$\mathit{res} : \GHom(\ZD,H) \to  \GHom(A,H)$ given by $\mathit{res}(x)= x|_{A}$ is thus well defined.

With this notation,
 $$C_3^d= \GHom(\ZD,\ZZ/3\ZZ).$$

Consider the space $\GHom(\ZD,\ZZ) \subset \ZZ^{\ZD}$ of graph homomorphisms from the Cayley graph of $\ZD$ to the Cayley graph of $\ZZ$, both with respect to the standard generators. Namely,
$$\GHom(\ZD,\ZZ) = \{ x \in \ZZ^{\ZD} |x_n -x_m | =1 \mbox{ whenver } \|n-m\|_1=1\}.$$

Since the Cayley graph of $\ZZ$ covers the Cayley graph $\ZZ/3 \ZZ$, it follows that $\GHom(\ZD,\ZZ)$ projects to $C_3^d$ via the following continuous shift-equivariant map $\pi:\GHom(\ZD,\ZZ) \to C_3^d$ defined by:
\begin{equation}
\pi(x)_n := x_n \mod 3.
\end{equation}

The following  observation is classical

 (see for instance  \cite[ Section 4.3]{schmidt_cohomology_SFT_1995}): 
\begin{prop}\label{prop:3CB_lift}
The map $\pi:\GHom(\ZD,\ZZ) \to C_3^d$ is surjective. Furthermore, if $x,y \in \GHom(\ZD,\ZZ)$ satisfy $\pi(x)=\pi(y)$ then there exists $m \in 3\ZZ$ so that $x_n = y_n +m$ for all $n \in \ZD$.
\end{prop}

Our goal in this section is to prove the following result:
\begin{thm}\label{thm:3CB_prime}
For any $d \ge 1$, the $\mathbb{Z}^d$ 3-colored $C_3^d$ chessboard is  \dprime.
\end{thm}

The case $d=1$  easily follows from Lind's work \cite{lind_SFT_factorization_1983}: Note that $h(C_3^1)=\log(2)$ and $2$ is a prime number (in particular an irreducible Perron number). Now observe that
$C_3^1$  is a mixing $\ZZ$-SFT and conclude it must be direct-prime.

For $d \ge 2$ our argument is based on the cohomology of shift action on $C_3^d$. We briefly recall definitions to make our exposition reasonably self-contained.
See \cite{schmidt_cohomology_SFT_1995} for various results and applications regarding  cohomology of $\ZD$ subshifts of finite type.

\begin{defn}
Let $(X,T)$ be a $\mathbb{Z}^d$-dynamical system.

\begin{enumerate}
\item{
A continuous cocycle  $(X,T)$  (abbreviated $T$-cocycle, or cocycle when $T$ is clear for the context) is a continuous function $c: X \times \mathbb{Z}^d \to \mathbb{R}$
satisfying
$$c(x,n+m) = c(x,m) + c(T^{m} x,n), ~\forall x \in X,~ m,n \in \ZD.$$
}
\item{ A \emph{trivial} cocycle is of the  from  $c(x,n)= \alpha(n)$, where $\alpha \in \mathit{Hom}(\mathbb{Z}^d,\mathbb{R})$.}
\item{
A $T$-coboundary is a cocycle $b:X \times \mathbb{Z}^d \to \mathbb{R}$ of the form
$b(x,n) = f(T^n x) - f(x)$, where $f \in C(X)$. }
\item{The \emph{cohomology group} of $(X,T)$ is the group of  cocycles modulo the subgroup of coboundaries, where the group operation is pointwise addition. We denote the  cohomology group by $H(X,T)$. When $X$ is a subshift and $T$ the shift action, we abbreviate this by $H(X)$.}
\end{enumerate}
\end{defn}

By Proposition \ref{prop:3CB_lift} the map  $\Ht:C_3^d \times \ZD \to \RR$ given by
\begin{equation}\label{eq:Ht_def}
\Ht(x,n) = \hat{x}_0 - \hat{x}_n,
\end{equation}
where $\hat{x} \in \GHom(\ZD,\ZZ)$ satisfies $\pi(\hat{x})=x$ is a well-defined cocycle. It is  not difficult to check that  $\Ht$ is not cohomologous to a trivial cocycle (see \cite{schmidt_fundamental_cocycles_1998}). We refer to $\Ht$ as the \emph{height cocycle}.

In the following we  consider $\ZZ^{d-1}$ as a subgroup of $\ZZ^d$, using the embedding
$$(n_1,\ldots, n_{d-1}) \mapsto (n_1,\ldots,n_{d-1},0).$$

The subshift $C_3^{d-1}$ is embedded as a sub-system of a $\ZZ^{d-1}$ sub-action of $C_3^d$ as follows: Consider the
subspace
\begin{equation}\label{eq:tilde_C_def}
\tilde{C}_3^d = \{ x\in C_3^{d}~:~ x_{n+e_d}=x_{n} +1 \mod 3 ~\forall n \in \ZD \}.
\end{equation}

The map $x \mapsto x|_{\ZZ^{d-1}}$ is a $\ZZ^{d-1}$-equivariant homeomorphism from $\tilde{C}_3^d$ to $C_3^{d-1}$.

\begin{lem}\label{lem:3CB_glue}
For any $x \in C_3^d$ and $N \in \mathbb{Z}$ there exist $y \in C_3^d$ and $z \in \tilde{C}_3^d$ so that
\begin{equation}\label{eq:y_def}
y_{n+ke_d} = \begin{cases}
 x_{n+ke_d} & k \ge N\\
z_{n+ke_d}   & k \le N
\end{cases},
 ~n \in \ZZ^{d-1},
\end{equation}
\end{lem}
\begin{proof}
For  $x \in C_3^d$, define $z \in \tilde{C}_3^d$  by
\begin{equation}\label{eq:z_def_ht}
z_{n + ke_d} := x_{n+Ne_d} + k-N \mod 3,
~ n \in \ZZ^{d-1}, ~ k \in \ZZ
\end{equation}
Check that indeed $z \in \tilde{C}_3^d$ and that the unique $y \in \{0,1,2\}^{\ZD}$ given by \eqref{eq:y_def} is a well defined point in $C_3^d$.

\end{proof}

The following is an adaptation of a result  from  \cite{schmidt_fundamental_cocycles_1998}:

\begin{prop}\label{prop:cohomology_X_3}
For any $d \ge 2$, any continuous locally constant cocycle on $C_3^d$ is cohomologous to a sum of trivial cocycle
 and a multiple of the height cocycle $\mathit{ht}$.
To be precise, for any locally constant cocycle $c:C_3^d \times \ZD \to \mathbb{R}$ there exist $a \in \mathbb{R}$,
$\alpha \in \Hom(\ZD,\mathbb{R})$ and a locally constant function $f:C_3^d \to \mathbb{R}$ so that
$$c(x,n) = a \cdot \Ht(x,n) + \alpha(n) + f(\sigma^n(x))-f(x).$$
\end{prop}

\begin{proof}
For $d=2$ this follows from \cite[ Theorem $7.1$]{schmidt_fundamental_cocycles_1998}, which identifies the so called ``fundamental cocycle'' for the two-dimensional $3$-colored chessboard $C_3^{(2)}$.
Instead of  trying to extend the  arguments of \cite{schmidt_fundamental_cocycles_1998} to higher dimensions, we will proceed by induction on $d \ge 2$, using  \cite[ Theorem $7.1$]{schmidt_fundamental_cocycles_1998} to start the induction.

Let $c:C_3^d \times \ZD \to \mathbb{R}$ be a locally constant cocycle.
Hence there exists   a finite set $F \Subset \ZD$ so that  so that $x|_{F}$ determines $c(x,e_i)$ for all $x \in C_3^d$ and $i=1,\ldots,d$.  
It follows that the restriction of $x|_{F+\ZZ^{d-1}}$ determines $c(x,n)$ for all $n \in \ZZ^{d-1} $.
Choose $x \in C_3^d$.
Apply  Lemma \ref{lem:3CB_glue} to find $y \in C_3^d$ and $z \in \tilde{C}_3^d$ so that $y|_{F+\ZZ^{d-1}}=x|_{F+\ZZ^{d-1}}$,
 and so that for all $k \ge N$,
$$ \sigma^{ke_d}(y)|_{F+\ZZ^{d-1}}= \sigma^{ke_d}(z)|_{F+\ZZ^{d-1}}.$$

It follows that for $n \in \ZZ^{d-1}$
$$c(x,n)=c(y,n)=c(y,2Ne_d)+c(\sigma^{2Ne_d}(y),n)+c(\sigma^{2Ne_d+n}(y),-2Ne_d),$$
so
\begin{equation}\label{eq:3CB_cocycle_reduction}
c(x,n)=c(y,2Ne_d)+ \tilde{c}(z,n) - c(\sigma^{n}y,2Ne_d),
\end{equation}
where $\tilde c: \tilde{C}^d_3 \times \ZZ^{d-1} \to \mathbb{R}$ is the restriction of $c:C^d_3 \times \ZD \to \mathbb{R}$.
Because the $\ZZ^{d-1}$ shift action on $\tilde{C}^d_3$ is isomorphic to $C^{d-1}_3$, by  induction on $d \ge 2$,
there exist $a \in \mathbb{R}$, $\alpha \in \Hom(\ZZ^{d-1}, \mathbb{R})$ and a locally constant function $\tilde f:\tilde{C}^d_3 \to \mathbb{R}$ so that

\begin{equation}\label{eq:tilde_c_def}
\tilde{c}(z,n) = a \Ht(z,n) + \alpha(n) + \tilde{f}(\sigma^n(z)) - \tilde{f}(z).
\end{equation}
From \eqref{eq:z_def_ht} we see that
$\Ht(z,n)=\Ht(\sigma^{Ne_d}(x),n)$ for $n \in \ZZ^{d-1}$.
Let $f(y) := c(y,3Ne_d)$. It follows from \eqref{eq:3CB_cocycle_reduction} and \eqref{eq:tilde_c_def} that for any $n \in \ZZ^{d-1}$ and any $x \in C_3^d$
\begin{equation}
c(x,n)= a\cdot \Ht(\sigma^{Ne_d}(x),n) + \alpha(n) + \tilde f(\sigma^n(x)) -\tilde f(x) + f(\sigma^n(x)) - f(x) .
\end{equation}
Now because
$$ \Ht(x,n) - \Ht(\sigma^{Ne_d}(x),n) = \Ht(x,Ne_d) -\Ht(\sigma^{n}(x),Ne_d),$$
we see that  indeed that there exists  $a_1 \in \mathbb{R}$, $\alpha_1 \in \Hom(\ZD,\mathbb{R})$  and a locally constant function $f_1:C_3^d \to \mathbb{R}$ so
that for any $x \in C_3^d$,
\begin{equation}\label{eq:3CB_coycle_subaction1}
c(x,e_i) = a_1 \Ht(x,e_i) + \alpha_1(e_i) + f_1(\sigma^{e_i}x)-f_1(x)
\end{equation}
for all $1\le i \le d-1$.

Repeat the above argument, this time embedding $\ZZ^{d-1}$ in $\ZD$ via
$$(n_1,\ldots,n_{d-1}) \mapsto (0,n_1,\ldots,n_{d-1})$$
to conclude that
there exists  $a_2 \in \mathbb{R}$, $\alpha_2 \in \Hom(\ZD,\mathbb{R})$  and a locally constant functions $f_2:C_3^d \to \mathbb{R}$ so
that for any $x \in C_3^d$
\begin{equation}\label{eq:3CB_coycle_subaction2}
c(x,e_i) = a_2 \Ht(x,e_i)  + \alpha_2(e_i) +f_2(\sigma^{e_i}x)-f_2(x)
\end{equation}
for all $2\le i \le d$.
The proof will be complete once we show that we can choose $a_1=a_2$ $\alpha_1=\alpha_2$ and  $f_1=f_2$.

Since $d \ge 3$, we choose  $2 \le i \le d-1$, and conclude from \eqref{eq:3CB_coycle_subaction1} and \eqref{eq:3CB_coycle_subaction2} that $a_1=a_2$
$\alpha_1(e_i) = \alpha_2(e_i)$ and $f_1(\sigma^{ke_i}x)-f_1(x)=f_2(k\sigma^{ke_i}x)-f_2(x)$ for all $k \in \NN$ and $x \in C_3^d$.
We can thus define $\alpha \in \Hom(\ZD,\mathbb{R})$ as follows:
$$\alpha(e_i) = \begin{cases}
\alpha_1(e_i) & 1 \le i \le d-1\\
\alpha_2(e_i) & i=d
\end{cases}.
$$
Because $C_3^{(d)}$ is topologically mixing, it follows that $f_1$ and $f_2$ differ by a constant,
and so indeed the cocycle $c:C_3^d \times \ZD \to \mathbb{R}$ is cohomologous to $a\cdot \Ht + \alpha$.
\end{proof}

The following lemma
 is a slight refinement of the fact that $C_3^d$ has a dense set of periodic points.

\begin{lem}\label{lem:3CB_densePP}
 For any infinite $A \subset \NN$, the set of points $x \in C_3^d$ whose stabilizer contains $k\ZD$ for some $k \in A$
is dense in $C_3^d$.
\end{lem}
\begin{proof}
By Proposition \ref{prop:3CB_lift} the statement will  follow once we prove that a corresponding set of points 
is dense in $\GHom(\ZD,\ZZ)$.

Indeed, for any $k \in \NN$, any $\hat{x} \in \GHom([-k,k]^d,\ZZ)$ is the restriction of a
$\hat{y} \in \GHom([-2k,2k],\ZZ)$ with ``flat boundary'', that is $|\hat{y}_n-\hat{y}_m| \le 1$ for any
$n,m\in \partial [-k,k]^d$, where $\partial [-k,k]^d := \{ m=(m_1,\ldots,m_d) \in \ZD~:~ \max_i |m_i| = k\}$.
A proof of can be found  for instance in \cite[ Lemma $7.3$]{changotia_meyerovitch_3CB}.
\end{proof}

The next lemma says that there only $3$ points in $C_3^d$ of ``maximal slope'':

\begin{lem}\label{lem:3CB_max_slope_points}
There exist precisely $3$ points $x \in C_3^d$ which 
satisfy
\begin{equation}\label{eq:max_ht}
\Ht(x,m) = m_1+m_2+\ldots + m_d
\end{equation}
for any $m= (m_1,\ldots, m_d) \in \ZD$.

Furthermore, if $x \in C_3^d$ has finite orbit and satisfies \eqref{eq:max_ht} for all $m$ in the stabilizer of $x$, then $x$ satisfies
 \eqref{eq:max_ht}  for all $m \in \ZD$.
\end{lem}
\begin{proof}
Suppose $\hat{x} \in \GHom(\ZD,\ZZ)$ and $(m_1,\ldots,m_d) \in \ZD$  are such that
\begin{equation}\label{eq:hat_x_max_slope}
\hat{x}_{(m_1,\ldots,m_d)} - \hat{x}_0 = m_1 + \ldots +m_d
\end{equation}

Because $\hat{x}_{n+e_i}-\hat{x}_n \le 1$ for all $n \in \ZD$,  it follows that $\hat{x}_{n+e_i} = \hat{x}_n +1$  whenever
$n$ and $n+e_i$ are on a shortest path connecting $0$ and $(m_1,\ldots,m_d)$ in the Cayley graph of $\ZD$.
Now suppose $x$ satisfies  \eqref{eq:max_ht} for all  $m= (m_1,\ldots, m_d) \in \ZD$ and let $\hat{x} \in \GHom(\ZD,\ZZ)$ be such that $\pi(\hat{x})=x$.
It follows that $\hat{x}_{n+e_i} = \hat{x}_n +1$ for all $n \in \ZD$ and $i=1,\ldots,d$. Thus the value of $x_0 \in \ZZ / 3\ZZ$ uniquely determines $x$. This proves the first part of the lemma.

Now if $x$ has finite orbit, $L$ is the stabilizer of $x$ and \eqref{eq:max_ht} holds for all $m \in L$ and
$\hat{x} \in \GHom(\ZD,\ZZ)$ satisfies $\pi(\hat{x})=x$, then $\hat{x}_{n+e_i} = \hat{x}_n +1$ for all $n \in \ZD$, because
there exists some $m \in L$ for which $n$ and $n+e_i$ are on some shortest path connecting $0$ and $m$.
This proves the second part of the lemma.
\end{proof}

Recall that an automorphism $\psi$ of a $\ZD$-topological dynamical system $(X,T)$ is  homeomorphism $\psi:X \to X$  satisfying $\psi(T^n(x))= T^n(\psi(x))$ for all $x \in X$, $n \in \ZD$. We denote  the group of automorphisms of $(X,T)$ by $\Aut(X,T)$.
\begin{lem}\label{lem:aut_on_3BC}
For any $\psi \in \Aut(C_3^d,\sigma)$, there exists $ u_\psi \in \{\pm 1\}$ such that
the cocycle $\Ht_{\psi}$ defined by $\Ht_\psi(x,n) := \Ht(\psi(x),n)$ is cohomologous to $u_\psi \cdot \Ht$.
\end{lem}
\begin{proof}

By Proposition \ref{prop:cohomology_X_3} above, there exists $a_\psi \in \RR$ so that $\Ht_\psi$ is cohomologous to $a_\psi \cdot \Ht + \alpha_\psi$, where $\alpha_\psi \in \Hom(\ZD,\mathbb{R})$ is a trivial cocycle.
Note that for any $n=(n_1,\ldots,n_d) \in \ZD$ such that $\|n\|_1$ is even
$$\max_{x \in C_3^d} \Ht(x, n) = 2\|n\|_1 \mbox{ and }\min_{x \in C_3^d}\Ht(x, n)=- 2\|n\|_1,  $$
So
$$\max_{x \in C_3^d} \Ht_\psi(x, n) = 2|a_\psi|\cdot\|n\|_1+\alpha_\psi(n)  \mbox{ and }\min_{x \in C_3^d} \Ht_\psi(x, n)= -2|a_\psi|\cdot\|n\|_1+ \alpha_\psi(n),  $$
On the other hand, since $\psi$ is surjective,
$$\max_{x \in C_3^d} \Ht(x, n) = \max_{x \in C_3^d} \Ht_\psi(x, n) \mbox{ and }\min_{x \in C_3^d}\Ht(x,n) = \min_{x \in C_3^d}\Ht_\psi(x,n),  $$
It follows that $a_\psi \in \{\pm 1\}$ and $\alpha_\psi(n)=0$.
\end{proof}

\noindent \emph{Concluding the proof of theorem \ref{thm:3CB_prime}}:

Suppose $C_3^d \cong Y \times Z$ is a non-trivial direct factorization.

By  Lemma \ref{lem:3CB_densePP} there exists a finite index subgroup $L \lhd \mathbb{Z}^d$  for which $(C_3^d)^L:= \bigcap_{n \in L}\{ x \in C_3^d:~ \sigma^n x = x \}$ is non-empty. Furthermore, we can choose $L$ so that $[\ZD:L]$ is odd. It follows that  $Z^L:= \bigcap_{n \in L}\{ z \in Z~:~ \sigma^n z = z \}$ is also non-empty. Choose $\hat{z} \in Z^L$, and a finite set $F_L \subset \ZD$ of representatives for $\ZD / L$. Define a cocycle $c:Y \times\mathbb{Z}^d \to \mathbb{R}$ by
$$c(y,n) := \sum_{m \in F_L}\Ht\left((y,\sigma^m \hat{z}),n\right),$$
where we naturally identify the pair $(y,\sigma^m \hat{z})$ is an element of $C_3^d$.
To check that $c$ is indeed a cocycle note that for any $n \in \ZD$,
$$\left\{\sigma^m\tilde z ~:~ m \in F_L \right\} = \left\{\sigma^{m+n}\tilde z ~:~ m \in F_L \right\},$$
because both are equal to the orbit of $\tilde z$ under $\sigma$.
Thus, for any $n_1,n_2 \in \ZD$
$$c(y,n_1+n_2) = \sum_{m \in F_L}\Ht\left((y,\sigma^m \hat{z}),n_1\right))+ \sum_{m \in F_L}\Ht\left((\sigma^{n_1}y,\sigma^{m+n_1} \hat{z}),n_2\right)= c(y,n_1)+ c(\sigma^{n_1}y,n_2)$$
Because  $(Y,\sigma)$ is a factor of $(C_3^d,\sigma)$, $c$ naturally lifts to a locally constant cocycle on $(C_3^d,\sigma)$.

From Proposition \ref{prop:cohomology_X_3} it  follows that there exists
$\phi = \phi_{\hat{z}} \in \Hom(\ZD,\RR)$ and $\alpha_{\hat{z}} \in \RR$ (both a priori depending on $\hat{z}$)
so that $c:C_3^d \times \ZD \to \RR$ is cohomologous to $\alpha_{\hat{z}} \Ht + \phi$.

Observe that for any $m \in \ZD$, the map $(y,z) \mapsto (y,\sigma^m(z))$ is an automorphism of $C_3^d$. It follows from Lemma \ref{lem:aut_on_3BC} that for any $m \in \ZD$ either
$\Ht((y,z),n)=\Ht((y,\sigma^m(z)),n)$ for all $y \in Y$, $z \in Z$ $n \in \ZD$ or
$\Ht((y,z),n)=-\Ht((y,\sigma^m(z)),n)$.This means that there is homomorphism $s \in \Hom(\ZD,\ZZ/2\ZZ)$ so that
$\Ht((y,z),n)= (-1)^{s(m)}\Ht((y,\sigma^m(z)),n)$.
Because $[\ZD : L]$ is odd and $Z^L$ non empty 
there is some $C \in \ZZ \setminus \{0\}$ and coboundary $b:C_3^d \times \ZD \to \mathbb{R}$
so that for any $z \in Z$ and $n \in \ZD$,
 $$\sum_{m \in F_L}\Ht((y,\sigma^m z),n)=  C \cdot \Ht((y, z),n)+ b((y,z),n)$$

Suppose first that there  \emph{exist} $L < \ZD$ and $\hat{z} \in Z^L$  as above so that $\alpha_{\hat{z}} \ne 0$, it follows that $\Ht:(Y \times Z) \times \ZD \to \mathbb{R}$ is cohomologous to a cocycle which only depends on $Y$, and so $\phi_{\hat{z}} = 0$ for all $\hat{z}$. Thus $\alpha_{z} = \alpha_{\hat{z}} \ne 0$  for \emph{any} point $z \in C_3^d$ with finite orbit.

Choose $k \in 3\ZZ$, and  let $\tilde{L}= k \ZD$. Let $\tilde{y} \in Y$, $\tilde{z} \in Z$ be such that $(\tilde{y},\tilde{z}) \cong  x \in C_3^d$ satisfies $\sigma_{ke_i} x= x$  and $\Ht(x,ke_i)=k$  for $i=1,\ldots,d$, as in Lemma \ref{lem:3CB_max_slope_points}.
Since $\Ht$ is cohomologous to cocycle which does not depend on $z$,
it follows that $\Ht((\tilde{y},z),ke_i)= k$  for all $z \in Z^{\tilde{L}}$. From Lemma \ref{lem:3CB_max_slope_points}
we conclude that there are at most $3$  points in $\bigcup_{L}Z^{L}$, where the union is over all finite index subgroups $L$ such that $[\ZD:L]$ is odd.
By Lemma a \ref{lem:3CB_densePP}
$\bigcup_{L}(C_3^d)^L$
is a dense in $C_3^d$. It follows that
$\bigcup_{L}Z^L$
is dense in $Z$, so $Z$ must be finite.
Because  $C_3^d$ is topologically mixing, it  has only trivial finite  factors.
This implies $Z$ is a trivial one-point system.

Otherwise, $\alpha_{\hat{z}} = 0$ for all $\hat{z} \in Z^L$ with $[\ZD:L]$ odd. In this case it follows  that $\Ht$ is cohomologous to a cocycle which only depends on $Z$. Replacing the roles of $Y$ and $Z$, we conclude using Lemma \ref{lem:3CB_densePP} as in the previous case that $Y$ is finite, hence trivial.

\section{Direct-primeness for Dyck shifts}
\label{sec:dyck_factorization}

Dyck shifts are a one parameter class of  non-sofic $\ZZ$-subshifts.
They were introduced by Krieger  in \cite{krieger_74}, as a counterexample to a conjecture of Weiss, and appeared in various papers in the literature since.
We now recall a definition of the Dyck shifts:

 Let $N>1$ be a natural number.
 Write $\Sigma_N = \left(\{\alpha_1,\ldots,\alpha_N\}\cup \{\beta_1,\ldots,\beta_N\}\right)$. Consider the monoid $M$ generated by $\Sigma_N \cup \{0\}$ subject to the following relations:
\begin{enumerate}
\item $\alpha_i \beta_i= 1$ for $i \in \{1,\ldots N\}$.
\item $\alpha_i \beta_j = 0$ for $i \ne j$, $i,j \in \{1,\ldots N\}$.
\end{enumerate}

The $N$-Dyck shift is defined by:
$$D_N = \{ x \in \Sigma_N^\ZZ ~:~ x_{n}\cdot x_{n+1}\cdot \ldots \cdot x_{n+k}\neq  0 ~\forall n\in \ZZ, k \in \NN\}.$$

Informally, if we think of $\{\alpha_1,\ldots,\alpha_N\}$ as $N$ types of ``left brackets'' and of $\{\beta_1,\ldots,\beta_N\}$ as $N$ corresponding ``right brackets'', $D_N$ consists of all bi-sequences with no ``mismatching pairs of brackets''.

In this section we prove the following:

\begin{thm}\label{thm:dyck_prime}
For any prime number $N$,  the $N$-Dyck shift $D_N$ is topologically direct prime.
\end{thm}

The assumption that $N$ is prime seems to be an artifact of the proof method.

Let us introduce auxiliary definitions.
Following \cite{krieger_74}, define two continuous shift-commuting maps
$\pi_+,\pi_-:D_N \to \{0,\ldots,N\}^\ZZ$ by
$$\pi_+(x)_n = \begin{cases}
i & x_n = \alpha_i\\
0 & x_n \in \{\beta_1,\ldots,\beta_N\}
\end{cases}
$$

$$\pi_-(x)_n = \begin{cases}
i & x_n = \beta_i\\
0 & x_n \in \{\alpha_1,\ldots,\alpha_N\}
\end{cases}
$$

Denote by $\nu_N$ the uniform Bernoulli measure on $\{0,\ldots,N\}^\mathbb{Z}$, which is uniquely defined by
$\nu_N ([a]_k) = \frac{1}{N^m}$ for all $a= (a_1,\ldots,a_m) \in \{0,\ldots,N\}^m$.
It was  observed  in \cite{krieger_74} that there is a shift-invariant Borel set $X_0 \subset\{0,\ldots,N\}^\mathbb{Z}$ of full $\nu_M$-measure so that any $x \in X_0$ has a unique pre-image  under $\pi_-$ and   a unique pre-image under $\pi_+$ .

A dynamical system is called \emph{intrinsically ergodic} if it admits a unique measure of maximal entropy.
One interesting feature of Dyck shifts, which was discovered in \cite{krieger_74}, is that they are not intrinsically ergodic:
\begin{thm}\label{thm:dyck_2_mme} \emph{(\cite[Theorem $3$]{krieger_74})}
For any $N \ge 2$ there exists precisely $2$ ergodic measures of maximal entropy $\mu_+$ and $\mu_-$ for the Dyck shift $D_N$. The measures $\mu_+$ and $\mu_-$ are the pull-back of the uniform Bernoulli measure on $\{0,\ldots,N\}^\mathbb{Z}$ via $\pi_{+}$ and $\pi_{-}$ respectively.
\end{thm}

Recall that a \emph{joining} of two probability preserving $\GG$-actions $(X,\mathcal{B},\mu,T)$ and $(Y,\mathcal{C},\nu,S)$ is a
probability measure $\lambda$ on $(X\times,\mathcal{B}\otimes \mathcal{C})$ which is $T \times S$-invariant and has $\mu= \lambda \circ \pi_X^{-1}$ and $\nu= \lambda \circ \pi_Y^{-1}$  where  $\pi_X :X\times Y \to X$ and $\pi_Y:X \times Y \to Y$ are the obvious projection maps. $(X,\mathcal{B},\mu,T)$ and $(Y,\mathcal{C},\nu,S)$ are \emph{disjoint} if the only joining of the systems is the independent joining $\lambda = \mu \times \nu$.

 We say that a pair of   probability preserving actions of an amenable group $\GG$  $(X,\mathcal{B},\mu,T)$ and $(Y,\mathcal{C},\nu,S)$ are \emph{intrinsically disjoint} if the independent joining is the only joining which maximizes the entropy.

We use this term to state the following simple Lemma:

\begin{lem}\label{lem:bernoullis_intrinsically_disjoint} \emph{ (Bernoulli transformations are ``pairwise intrinsically disjoint'')}
Let $(X,\mathcal{B},\mu,T)$ and $(Y,\mathcal{C},\nu,S)$ be Bernoulli transformations with finite entropy, and $\lambda$ a joining of the two such that
$h_\lambda(T \times S) = h_\mu(T) + h_\nu(S)$. Then $\lambda = \mu \times \nu$ is the independent joining.
\end{lem}

\begin{proof}
Let $\alpha \subset \mathcal{B}$ and $\beta \subset \mathcal{C}$ be finite partitions which are independent generators for $T$ and $S$ respectively. This is equivalent to the statement that for any $n \in \NN$,
$$\frac{1}{n}H_\mu(\bigvee_{k=0}^{n-1} T^{-k}\alpha) = H_\mu(\alpha) =h_\mu(X,T),$$
$$\frac{1}{n}H_\nu(\bigvee_{k=0}^{n-1}T^{-k}\beta) = H_\nu(\beta) =h_\mu(Y,S).$$
Since $\alpha \vee \beta$ is a two-sided generator, it follows that
$$h_\lambda(X\times Y, T \times S)= \lim_{n \to \infty}\frac{1}{n}H_\lambda(\bigvee_{k=0}^n T^{-k}\alpha \vee \bigvee_{k=0}^n T^{-k}\beta).$$

By subadditivity,  for any $m \ge 1$,
$$\frac{1}{m}H_\lambda(\bigvee_{k=0}^{m-1} T^{-k}\alpha \vee \bigvee_{k=0}^{m-1} T^{-k}\beta) \ge \lim_{n \to \infty}\frac{1}{n}H_\lambda(\bigvee_{k=0}^{n-1} T^{-k}\alpha \vee \bigvee_{k=0}^{n-1} T^{-k}\beta).$$

Thus $h_\lambda(X\times Y, T \times S)= H_\nu(\beta) + H_\mu(\alpha)$ if and only if for every $n \ge 1$,
$$ \frac{1}{n}H_\lambda(\bigvee_{k=0}^n T^{-k}\alpha \vee \bigvee_{k=0}^{n-1} T^{-k}\beta) = H_\nu(\beta) + H_\mu(\alpha).$$
This equality can hold if and only if  $\{ T^{-k}\alpha, T^{-j}\beta\}_{k,j \in \ZZ}$ are jointly independent. It follows that $\lambda$ is the independent joining .
\end{proof}
\begin{lem}\label{lem:dyck_factor_interisic_ergodic}
Suppose $D_N \cong Y \times Z$ is a direct topological factorization, then precisely one of the direct factors $Y$ and $Z$ is intrinsically ergodic and the other has precisely two ergodic measures of maximal entropy.
\end{lem}

\begin{proof}
Denote by $\nu_+$ and $\eta_+$ the projection of $\mu_+$ onto $Y$ and $Z$ respectively. Since the $(D_N,\sigma,\mu_+)$ is isomorphic to a Bernoulli shift, it follows that so are $(X,\sigma,\nu_+)$ and $(Y,\sigma,\eta_+)$. We have:
$$h_{\mu_+}(D_N,\sigma) \le h_{\nu_+}(Y,\sigma)+h_{\eta_+}(Z,\sigma) = h_{\nu_+ \times \eta_+}(Y \times Z, \sigma).$$
Since $\mu_+$ is  a measure of maximal entropy, the inequality must be an equality, and so by Lemma \ref{lem:bernoullis_intrinsically_disjoint} $\mu_+ = \nu_+ \times \eta_+$. Similarly, $\mu_- = \nu_- \times \eta_-$.
Since the four combinations $\nu_{\pm} \times \eta_{\pm}$ give precisely two ergodic measures on $D_N$, it follows that either $\nu_+ \ne \nu_-$ and $\eta_+ = \eta_-$ or vice versa.
\end{proof}

\begin{remark} In general, any direct factor of an intrinsically ergodic system is  intrinsically ergodic.
However, there are intrinsically ergodic homeomorphisms  $T$ and $S$ such that $T \times S$ is not intrinsically ergodic (for instance, this is the case if the measures of maximal entropy for $T$ and $S$ have a common zero-entropy factor).
\end{remark}

 It is obvious that any two systems admitting a non-trivial common zero entropy factor are not intrinsically disjoint.

\begin{question}
Are any two $K$-systems intrinsically disjoint?
\end{question}

From now on we assume $D_N \cong Y \times Z$ is a direct topological factorization, realized by  a shift-equivariant surjective homeomorphism $\Phi:Y \times Z \to D_N$. By the preceding lemma, we also assume  without loss of generality that $Z$ is intrinsically ergodic.
Our goal is to show $Z$ is the trivial one point system.
Denote by $\nu_+$ and $\nu_-$ the projections of $\mu_+$ and $\mu_-$ onto $Y$ and let $\eta$ denote the unique measure of maximal entropy for $Z$. It follows from Lemma \ref{lem:dyck_factor_interisic_ergodic} that $\mu_+ \cong \nu_+ \times \eta$ and $\mu_- \cong \nu_- \times \eta$.

For a subshift $(X,\sigma)$ and $n \in \NN$, let $X^{(n)}:=\{ x \in X~:~ \sigma^n(x)=x\}$ denote the $n$-periodic points of $X$.

\begin{lem}\label{lem:dyck_local_entropy_pp}
For any $n \in \NN$ and  $x \in D_N^{(n)}$ the limits
\begin{equation}\label{eq:dyck_local_entropy}
h_{\pm}(x):= \lim_{k \to \infty} \frac{1}{k} \log \mu_\pm ([x_1,\ldots,x_k])
\end{equation}
exist and are given by:
\begin{equation}\label{eq:h_pm_A_formula}
h_\pm(x) = \log(N+1) \pm A(x)\log(N)
\end{equation}
where
\begin{equation}\label{eq:A_def}
A(x) := \min\left(0,\frac{1}{n}\left(\left|\{1\le  j \le n ~:~ \pi_+(x)_j=0\}\right| - \left|\{1 \le j \le k ~:~ \pi_+(x)_j\ne 0\}\right|\right)\right)
\end{equation}

\end{lem}
\begin{proof}
Check directly from the definition of $\mu_+$ that
$$\mu_+\left([a_1,\ldots,a_k] \right) =  (\frac{1}{N+1})^k \cdot \frac{1}{N^A},$$
where
$$A:= \min\left(0,\#\left\{ 1\le t \le k~:~ a_t \in \{\beta_1,\ldots,\beta_N \}\right\}- \#\left\{ 1\le t \le k~:~ a_t \in \{\alpha_1,\ldots,\alpha_N \} \right\}\right)$$
is the number of ``unmatched $\beta_j$'s''.
The formula for $h_+(x)$ follows directly by setting $a_i=x_i$, taking logarithm, dividing by $k$ and taking the limit $k \to \infty$. The formula for $h_{-}(x)$ follows by symmetry.
\end{proof}

\begin{lem}\label{lem:dyck_local_entropy_factor}
There exists a sequence of integers $\{k_j\}_{j=1}^\infty$ with $\lim_{j \to \infty}k_j =\infty$ so that
for any $n \in \NN$ and  $y\in Y^{(n)}$ and $z \in Z^{(n)}$ the limits
$$h_{\pm}(y):= \lim_{j \to \infty} \frac{1}{k_j} \log\nu_\pm ([y_1,\ldots,y_{k_j}])$$
and
$$h(z):= \lim_{k \to \infty} \frac{1}{k_j} \log \eta ([y_1,\ldots,y_{k_j}])$$
exist and satisfy
\begin{equation}\label{eq:dyck_local_entropy_abramov_f}
h_{\pm}(y,z) = h_\pm(y) + h(z)
\end{equation}
Furthermore,
$h(z)= h(Z)$ is independent of $z \in Z^{(n)}$.
\end{lem}
\begin{proof}
Fix an integer $M$ big enough 
so that the $0$ coordinate of both $\Phi$ and $\Phi^{-1}$ are  determined by the coordinates $[-M,M]$.
Let $x \in D_N^{(n)}$ and write $x=\Phi(y,z)$ with $y \in Y^{(n)}$ and $z \in Z^{(n)}$.
Write $$a^{\pm}_k(x) = -\frac{1}{k} \log \mu_\pm ([x_1,\ldots,x_k]) ,$$
$$b^{\pm}_k(y) = -\frac{1}{k} \log \nu_\pm ([y_1,\ldots,y_k]) ,$$
$$c_k(z) = -\frac{1}{k} \log\eta ([z_1,\ldots,z_k])$$

Since $\mu_\pm = (\nu_\pm \times \eta)\circ \Phi^{-1}$,  for all $k > 2M$,
\begin{equation}\label{eq:dyck_ll_addition}
 b^{\pm}_{k-M-n}(y)+c_{k-M-n}(z) \le a^{\pm}_k(x) \le  b^{\pm}_{k+M+n}(y)+c_{k+M+n}(z) .
\end{equation}
We know that the sequences $\{a_k^{\pm}(x)\}_{k \ge 1}$,
$\{b_k^{\pm}(y)\}_{k \ge 1}$ and $\{c_k(z)\}_{k \ge 1}$ are all non-negative.
By Lemma \ref{lem:dyck_local_entropy_pp} the sequence $\{a_k^{\pm}(x)\}_{k \ge 1}$ converges for any $x \in \bigcup_{n \in \NN}D_N^{(n)}$.
It follows that the sequences $\{b_k^{\pm}(y)\}_{k \ge 1}$ and $\{c_k(z)\}_{k \ge 1}$ are bounded. Thus for one particular $z \in \bigcup_{n \in \NN}Z^{(n)}$ there is a subsequence $\{k_j\}_{j=1}^\infty$  along which $\{c_k(z)\}_{k \ge 1}$  converges. By \eqref{eq:dyck_ll_addition} $\{b_{k_j}^{\pm}(y)\}_{j \ge 1}$ converges along this same subsequence for \emph{any} $y \in \bigcup_{n \in \NN}Y^{(n)}$. Again by \eqref{eq:dyck_ll_addition} it follows that $\{c_{k_j}(z)\}_{j \ge 1}$ converges for \emph{any} $z \in \bigcup_{n \in \NN}Z^{(n)}$.
The formula \eqref{eq:dyck_local_entropy_abramov_f} follows directly from \eqref{eq:dyck_ll_addition} by taking a limit along the sequence $\{k_j\}$.

It remains to show that $h(z)=h(\hat{z})$ for all $z,\hat{z} \in \bigcup_{n \in \NN}Z^{(n)}$. Suppose otherwise, $h(z) < h(\hat{z})$. It follows that for all $y \in \bigcup_{n \in \NN}Y^{(n)}$
$$h_+(\Phi(y,z)) < h_+(\Phi(y,\hat{z})),$$
and
$$h_-(\Phi(y,z)) < h_-(\Phi(y,\hat{z})).$$

In particular there exist $x,\hat{x} \in \bigcup_{n \in \NN}D_N^{(n)}$ with
$$\min\{ h_-(x),h_+(x)\} < \min\{ h_-(\hat{x}),h_+(\hat{z})\}.$$
By lemma \ref{lem:dyck_local_entropy_pp} $\min\{ h_-(x),h_+(x)\} = \log(N+1)$ for all $x \in \bigcup_{n \in \NN}D_N^{(n)}$, which is a contradiction.
\end{proof}

For $c \in [-1,1]$ and $n \in \NN$ let
$$D_N^{(n,c)}:= \left\{x \in D_N ~:~ \sigma_N x =x ,~ h_+(x)-h_-(x) = c\log(N) \right\}.$$
similarly, set:
$$Y^{(n,c)}:= \left\{y \in Y ~:~ \sigma_N y =y ~, h_+(y)-h_-(y) = c\log(N) \right\}.$$

By Lemma \ref{lem:dyck_local_entropy_pp}, we have
$$D_N^{(n,c)}:= \left\{x \in D_N ~:~ \sigma_N x =x , A(x) = c \right\},$$
where $A(x)$ is given by \eqref{eq:A_def}.
In other words, $D_N^{(n,c)}$ is the set of legal bi-infinite sequences in $D_N^{(n)}$ 
in which the number of left brackets  minus the number of right-brackets is $cn$.
An elementary calculation shows that for $ -n < j < n$ such that $ n-j = 0 \mod 2$:

\begin{equation}\label{eq:dyck_pp_count}
|D_n^{(n,\frac{j}{n})}| = {n \choose \frac{n}{2} + \frac{j}{2}}N^{\frac{n}{2}+|\frac{j}{2}|}
\end{equation}

The term ${n \choose \frac{n}{2} + \frac{j}{2}}$ in \eqref{eq:dyck_pp_count} corresponds to selecting the locations of the left brackets within the cycle. The term $N^{\frac{n}{2}+|\frac{j}{2}|}$ corresponds to selecting the ``types'' of left-brackets independently, or selecting the types of right-brackets, according to the sign of $j$.

In particular,

$$|D_N^{(n,1)}|=|D^{(n,-1)}| =  N^n \mbox{ and }
|D_N^{(2n,0)}|={2n \choose n}N^n$$

Our next step in the proof of Theorem \ref{thm:dyck_prime}, is to show that for prime $N$,  $|Z^{(N^n)}|=1$ for all $n$.
A version of the following argument  appears in \cite[Theorem $7$ ]{lind_entropies_markov_shifts_1984} which is the case $d=1$ of Theorem \ref{thm:full_shift_direct_prime}:

By Lemma \ref{lem:dyck_local_entropy_factor},
\begin{equation}
|D_N^{(n,c)}| = |Y^{(n,c)}|\cdot |Z^{(n)}|
\end{equation}

We now assume $N$ is prime. 
For $k= 0,1,2,\ldots$ we have
\begin{equation}\label{eq:dyck_N_power_pp}
|Y^{(N^k,1})|\cdot|Z^{(N^k)}|=|D_{N}^{(N^k,1)}|=N^{N^k}.
\end{equation}
Because $N$ is prime it follows that both $|Y^{(N^k,1})|$ and $|Z^{(N^k)}|$ are non-negative integer powers of $N$.
In particular for $k=0$ we have $N = |Y^{(n,1)}|\cdot |Z^{(1)}|$.
Thus, either $|Z^{(1)}|= N$ or $|Z^{(1)}|= 1$.

Suppose first that $|Z^{(1)}|=N$, and so $|Y^{(1,1)}|=1$. For $n \in \NN$, denote by $Y^{(n)}_*$ the set of points whose \emph{least period} is $n$. It follows that $Y^{(n,1)}= \biguplus_{ m \mid n} Y^{(m,1)}_*$. Also note that $m$  divides $|Y^{(m,1)}_*|$, since $Y^{(m,1)}_*$ is a disjoint union of orbits each of which has cardinality $m$.
Thus, since $N$ is prime the only divisors of $N^k$ are $N^j$ for $j=0,\ldots, k$ so
$$|Y^{(N^k,1)}|= |Y^{(1,1)}| + \sum_{j =0}^k |Y^{(N^j,1)}_*| \equiv  |Y^{(1,1)}| \equiv 1 \mod N.$$
On the other hand, $|Y^{(N^k,1)}|$ divides $|D^{(N^k,1)}|$ and $|D^{(N^k,1)}| = N^{N^k}$  by  \eqref{eq:dyck_N_power_pp}. Thus, since we assume $N$ is prime we have that $|Y^{(N^k)}|= N^l$ for some $l=0,\ldots,N^k$.
It follows that $N^l \equiv 1 \mod N$ thus $l=0$ so $|Y^{(N^k,1)}|=1$ for all $k$.
It follows that $|Z^{(N^k)}|=|D_{N}^{(N^k,1)}|$ for all $k$. On the other hand, we have $$|D^{(N^k,1)}| \cdot |Y^{(N^k,0)}|  = |Z^{(N^k)}| \cdot |Y^{(N^k,0)}| = |D_N^{(N^K,0)}|$$
 so we obtain $|D_{N}^{(N^k,1)}| \le |D_{N}^{(N^k,0)}|$,
which by \eqref{eq:dyck_pp_count} is false for sufficiently large $k$.

We conclude that $|Z^{(1)}|= 1$. Repeating the above argument with $Y$ and $Z$ interchanged, it follows that $|Z^{(N^k)}|=1$ for all $k$, and in particular,
\begin{equation}\label{eq:dyck_direct_factor_pp_growth}
\liminf_{n \to \infty} \frac{1}{n}|Z^{(n)}| = 0
\end{equation}

We will now show this implies $h(Z)=0$:

\begin{lem}\label{lem:pp_entropy_dirct_factor}
Suppose the entropy of a subshift $X$ is determined by the growth rate  of the periodic points, in the sense that
$$\lim_{n \to \infty} \frac{1}{n}|X^{(n)}| = h(X),$$
Then for any direct factor $Z$ of $X$ the entropy is determined by the growth rate  of its periodic points:
$$\lim_{n \to \infty} \frac{1}{n}|Z^{(n)}| = h(Z),$$
and in particular the limit on the left-hand side exists.
\end{lem}
\begin{proof}
Note that for any subshift $Y$ and any $n \in \NN$,
$$| \{ y_{\{1,\ldots,n\}}~:~ y \in Y^{(n)}\}| \le | \{ y_{\{1,\ldots,n\}}~:~ y \in Y\}|.$$
It is well known (and easily verified from the definitions) that
for any subshift the exponential growth rate of the periodic points is bounded above by the topological entropy:
\begin{equation}\label{eq:pp_lower_bound_entropy}
\limsup_{n \to \infty} \frac{1}{n}\log|Y^{(n)}| \le h(Y).
\end{equation}
Suppose $X \cong Y \times Z$. It follows that $h(X) = h(Y) +h(Z)$ . If $X$ satisfies the assumption of the lemma then
$$ h(X)= \lim_{n \to \infty} \frac{1}{n} \left(\log|Y^{(n)}| + \log|Z^{(n)}|\right),$$
It follows from \eqref{eq:pp_lower_bound_entropy} that
$$\liminf_{n \to \infty} \frac{1}{n}\log|Z^{(n)}| \ge h(Z).$$
Again by \eqref{eq:pp_lower_bound_entropy} applied to $Z$, we conclude that
$$\lim_{n \to \infty} \frac{1}{n}|Z^{(n)}| = h(Z).$$
\end{proof}

\begin{lem}\label{lem:dyck_pp_entropy}
The topological entropy of any direct factor $Z$ of the $N$-Dyck shift is determined by the growth rate of its periodic points:
$$\lim_{n \to \infty} \frac{1}{n}|Z^{(n)}| = h(Z)$$
\end{lem}
\begin{proof}
The fact that the topological entropy of $D_N$  is equal to $\lim_{n \to \infty} \frac{1}{n}|D_N^{(n)}|$ is a particular case of  \cite[Proposition $3.1$ ]{krieger_zeta_markov_dyck_2011}, which gives the corresponding result for a bigger family of subshifts.
This also follows by a direct computation  of the limit
$\lim_{n \to \infty} \frac{1}{n}|D_N^{(n)}| = \log(N+1) = h(D_N)$, using \eqref{eq:dyck_pp_count}.
The last equality holds because $h_{\mu_+}(D_N)=h_{\mu_-}(D_N)=\log(N+1)$ by isomorphism to the Bernoulli $(N+1)$-shift as in \cite{krieger_74}, or by a direct computation.
The proof now follows by Lemma \ref{lem:pp_entropy_dirct_factor}.
\end{proof}

We conclude that $h(Z)=0$. This implies $Z$ is a trivial $1$-point subshift as follows:

Recall that a topological dynamical system has \emph{completely positive entropy} if its only zero-entropy factor  is
the trivial factor \cite{blanchard_cpe}. 

\begin{lem}\label{lem:dyck_cpe}
The $N$-Dyck shift $D_N$ has completely positive entropy. 
\end{lem}
\begin{proof}
The $N$-Dyck shift $D_N$ is a coded system in the sense of \cite{blanchard_hansel_coded_system_1986}. See the remark in  \cite[Section $2.1$]{meyerovitch_dyck}.
By \cite{blanchard_hansel_coded_system_1986} any coded system has completely positive entropy.
The last result follows by observing that a non-trivial factor of a coded system is itself a coded system, thus has positive entropy.
\end{proof}

\begin{remark}
It is possible complete the proof Theorem \ref{thm:dyck_prime} without using  Lemma \ref{lem:dyck_cpe} and the notion of  ``completely positive entropy''. An alternative argument is to prove the set
$\bigcup_{k =1}^\infty D_N^{(N^k)}$ of $N^k$-periodic points is dense in $D_N$. This property passes to direct factors.
\end{remark}

\begin{remark}
``Entropy like'' quantities such as $h_+$ and $h_-$ defined above can lead to invariants associated to periodic points of certain types of subshifts. Another kind of invariant associated to periodic points are
``multiplier'' as in  \cite{hamachi_inoue_krieger}. These apply to periodic points of certain types of subshifts, including Dyck shifts. It seems plausible that using the technology of multipliers and the semi-group invariant introduced in \cite{MR1756982} it is possible to obtain more general  results about direct factorizations of Markov-Dyck shifts and more generally subshifts with Krieger's property $(A)$,  introduced in \cite{MR1756982}.
\end{remark}

\bibliographystyle{abbrv}
\bibliography{subshift_product}
\end{document}